\documentclass[12pt]{article}

\usepackage[centertags]{amsmath}
\usepackage{amsfonts}
\usepackage{amssymb}
\usepackage{mathrsfs}
\usepackage{setspace}
\usepackage{amsthm}
\usepackage{graphicx}
\usepackage{authblk}

\newtheorem{thm}{Theorem}[section]

\doublespace

\begin{document}

\title{Multi-dimensional Weiss operators}

\author[1,2]{Sergey Borisenok}
\author[3,*]{M. Hakan Erkut}
\author[1]{Ya\c{s}ar Polato\u{g}lu}
\author[1]{Murat Demirer}

\date{}

\affil[1]{Department of Mathematics and Computer Science,
\.{I}stanbul K\"{u}lt\"{u}r University, Bak\i rk\"{o}y 34156,
\.{I}stanbul, Turkey, email: s.borisenok@iku.edu.tr,
y.polatoglu@iku.edu.tr, m.demirer@iku.edu.tr} \affil[2]{Abdus Salam
School of Mathematical Sciences, Government College University, 68B
New Muslim Town 54600, Lahore, Pakistan, email: borisenok@gmail.com}
\affil[3]{Department of Physics, \.{I}stanbul K\"{u}lt\"{u}r
University, Bak\i rk\"{o}y 34156, \.{I}stanbul, Turkey, email:
m.erkut@iku.edu.tr} \affil[*]{{\bf Corresponding Author's Full
Postal Address: M. Hakan Erkut, \.{I}stanbul K\"{u}lt\"{u}r
\"{U}niversitesi, Fen-Edebiyat Fak\"{u}ltesi, Fizik
B\"{o}l\"{u}m\"{u}, Atak\"{o}y Kamp\"{u}s\"{u}, Bak\i rk\"{o}y
34156, \.{I}stanbul, Turkey. Mobile Phone Number: (+90) 532 740
7136. Fax Number: (+90) 212 661 9274. E-mail: m.erkut@iku.edu.tr}}

\maketitle

\begin{abstract}
We present a solution of the Weiss operator family generalized for
the case of $\mathbb{R}^{d}$ and formulate a $d$-dimensional
analogue of the Weiss Theorem. Most importantly, the generalization
of the Weiss Theorem allows us to find a sub-set of null class
functions for a partial differential equation with the generalized
Weiss operators. We illustrate the significance of our approach
through several examples of both linear and non-linear partial
differential equations.
\end{abstract}



\section{Introduction}

To investigate the integrability of nonlinear partial differential
equations, many different methods have been developed, among them is
Painlev\'e test \cite{WTC83}. In the framework of this approach, the
class of one-dimensional ('scalar') differential operators was
defined first in \cite{Weiss84}, and in \cite{Weiss86} has been
applied to solitonic-type PDEs.

The family of Weiss operators performs a special kind of ordinary
derivative operators of integer order $n$ ($n>0$, and for each order
only one such an operator exists), and the general solution for each
Weiss operator can be found in a very simple form. Following
\cite{Weiss84}, let's define for a differentiable scalar function
$\phi (x)$ the class of factorized differential operators as
\begin{equation}
L_{n+1}=\prod_{j=0}^{n}\left[ \frac{d}{dx}+\left( j-\frac{n}{2}\right) V%
\right] =\left( \frac{d}{dx}-\frac{n}{2}V\right) \cdot ...\cdot \left( \frac{%
d}{dx}+\frac{n}{2}V\right) \ ,  \label{1}
\end{equation}%
where $V=\phi _{xx}/\phi _{x}$ is the so-called pre-Schwarzian. For
$n=0$ equation (\ref{1}) produces the ordinary derivative
$L_{1}=d/dx$; for $n=1$ we get the Schr\"{o}dinger operator
\begin{equation*}
L_{2}=\left( \frac{d}{dx}-\frac{1}{2}V\right) \left( \frac{d}{dx}+\frac{1}{2}%
V\right) =\left( \frac{d}{dx}\right) ^{2}+\frac{1}{2}S\ ;
\end{equation*}%
for $n=2$, we get the Lenard operator
\begin{equation*}
L_{3}=\left( \frac{d}{dx}-V\right) \frac{d}{dx}\left( \frac{d}{dx}+V\right)
=\left( \frac{d}{dx}\right) ^{3}+2S\frac{d}{dx}+S_{x}\ ,
\end{equation*}%
where by $S$ we denote the Schwarzian of the function $\phi (x)$
\begin{equation}
S=V_{x}-\frac{1}{2}V^{2}=\left( \frac{\phi _{xx}}{\phi _{x}}\right) _{x}-%
\frac{1}{2}\left( \frac{\phi _{xx}}{\phi _{x}}\right) ^{2}\ .  \label{2}
\end{equation}

The important result of \cite{Weiss84} is related to the null space of Weiss
operators:

\begin{thm}
The null space of the operator family (\ref{1}) is given by the set of
linearly independent $n+1$ solutions $\{\phi _{x}^{-n/2}\phi ^{k}\}$; $%
k=0,1,2,...,n$, i.e. for every $k$
\begin{equation}
L_{n+1}[\phi _{x}^{-n/2}\phi ^{k}]=0\ .  \label{3}
\end{equation}
\end{thm}

In this paper we propose the generalization of Weiss family
$L_{n+1}$ for multi-dimensional case of $\mathbb{R}^{d}$. We
formulate a $d$-dimensional analogue of Theorem 1.1 for a special
class of partial differential operators that provides us with a
sub-set of null class functions for a partial differential equation
in terms of the generalized Weiss operators. Our generalization
therefore extends the role played by the Weiss operators in the
construction of some special class of ordinary differential
equations to partial differential equations as well. We illustrate
the significance of our approach through several examples of both
linear and non-linear partial differential equations.

\section{Generalized Weiss operators in $\mathbb{R}^{d}$}

Let's define for the vector $\boldsymbol{x}=(x_{1},x_{2},...,x_{d})\in
\mathbb{R}^{d}$ the linear differential operator:
\begin{equation}
D=\sum_{i=1}^{d}a_{i}(\boldsymbol{x})\frac{\partial }{\partial x_{i}}
\label{4}
\end{equation}%
with differentiable scalar functions $a_{i}(\boldsymbol{x})$. Note that for $%
a_{i}(\boldsymbol{x})=x_{i}$ the definition (\ref{4}) becomes the so-called
homogeneous operator.

Definition (\ref{4}) implies for the power $m$ of some scalar function $%
A(\boldsymbol{x})$:
\begin{equation}
D(A^{m}(\boldsymbol{x}))=\sum_{i=1}^{d}a_{i}(\boldsymbol{x})\frac{\partial
A^{m}(\boldsymbol{x})}{\partial x_{i}}=mA^{m-1}(\boldsymbol{x})DA(%
\boldsymbol{x})\ .  \label{5}
\end{equation}%
Also for a given differentiable function $\phi (\boldsymbol{x})$ we define
with (\ref{4}) a generalized pre-Schwarzian:
\begin{equation}
V=\frac{D^{2}\phi (\boldsymbol{x})}{D\phi (\boldsymbol{x})}\ .  \label{6}
\end{equation}%
The class of differential operators in $d$-dimensional space, given by:
\begin{equation}
L_{n+1}=\prod_{j=0}^{n}\left[ D+\left( j-\frac{n}{2}\right) V\right] =\left(
D-\frac{n}{2}V\right) \cdot ...\cdot \left( D+\frac{n}{2}V\right) \ ,
\label{7}
\end{equation}%
includes Weiss' operators (\ref{1}) as a particular case for $d=1$ and $%
a_{1}=1$. The index $n+1$ in (\ref{7}) denotes the order of the differential
operator $L_{n+1}$.

If, for instance, for $d=1$ we put $x\equiv x$, $a(x)\equiv a_1(x_1)$, then $%
V=a_x+a\phi _{xx}/\phi _x$ and
\begin{equation*}
L_2=a^2\frac{d^2}{dx^2}+aa_x\frac{d}{dx}+\frac{1}{2}aV_x-\frac{1}{4}V^2\ .
\end{equation*}
Also, if we suggest the particular case
\begin{equation*}
a_i(\boldsymbol{x})=F(\boldsymbol{x})c_i \ ,
\end{equation*}
where $c_i \not= 0$ are constants, then defining
\begin{equation*}
\xi =\sum_{i=1}^{d}\frac{x_i}{c_i}\ ,
\end{equation*}
we get
\begin{equation*}
\frac{\partial }{\partial \xi }=\sum_{i=1}^{d}c_i\frac{\partial}{\partial x_i%
}
\end{equation*}
and $D=F(\boldsymbol{x})\partial /\partial \xi$, that allows to
present the operators in one-dimensional form, i.e. $D\phi =F\phi
_\xi$,
\begin{equation*}
V=F_{\xi} +F\frac{\phi _{\xi \xi}}{\phi _{\xi}} \ .
\end{equation*}
The 'degenerated' case $F=1$ corresponds to (\ref{1}).

Now we can formulate the following theorem.

\begin{thm}
A set of linearly-independent functions:
\begin{equation}
\{(D\phi (\boldsymbol{x})^{-n/2}\phi ^{k}(\boldsymbol{x})\ ,\
k=0,1,2,...,n\}\ ,  \label{8}
\end{equation}%
forms a null function
\begin{equation}
f_{n+1}(\boldsymbol{x})=[D\phi (\boldsymbol{x})]^{-n/2}\sum_{k=0}^{n}c_{k}%
\phi ^{k}(\boldsymbol{x})\ ,  \label{9}
\end{equation}%
where $\{c_{k}\}$ is a set of constants, such that $f_{n+1}(\boldsymbol{x})$
satisfies the equation:
\begin{equation}
L_{n+1}f_{n+1}(\boldsymbol{x})=0\ .  \label{10}
\end{equation}
\end{thm}

\begin{proof}
Let's check how the operator $L_{n+1}$ acts on $D\phi ^{-n/2}\phi
^k$. Its first (rightmost) bracket using (\ref{5}) and (\ref{6})
produces
$$
\left( D+\frac{n}{2}V\right) (D\phi )^{-n/2}\phi ^k= (D\phi
)^{-n/2}\phi ^k \left[
-\frac{n}{2}\frac{D^2\phi}{D\phi}+k\frac{D\phi}{\phi}+\frac{n}{2}V
\right] =k(D\phi )^{-n/2+1}\phi ^{k-1}\ .
$$
The next bracket produces
$$
\left( D+\left(n-1-\frac{n}{2}\right) V\right) k(D\phi )^{-n/2+1}\phi ^{k-1} =k(k-1)(D\phi )^{-n/2+2}\phi ^{k-2}\ ,
$$
and so on. Each operator bracket from (\ref{7}) increases the power of $D\phi$
and decreases the power of $\phi$ by one. After the application of all $n+1$ brackets we end up with
$$
L_{n+1}(D\phi )^{-n/2}\phi ^{k}=k(k-1)(k-2)\cdot ...\cdot (k-n)(D\phi )^{-n/2+n+1}\phi ^{k-n-1} \ .
$$
But $k=0,1,2,...,n$, and one of the factors $k$, $k-1$, ... , $k-n$ is zero, thus,
\begin{eqnarray}
\label{proof1}
L_{n+1}(D\phi )^{-n/2}\phi ^{k}=0 \ .
\end{eqnarray}
As each operator bracket $(D +mV)$ is a linear differential
operator, (\ref{proof1}) can be applied to each term of the linearly
independent set $\sum _k (D\phi )^{-n/2}\phi ^{k}$, that is
$$
L_{n+1}\sum _{k=0}^{n}(D\phi )^{-n/2}\phi ^{k}=0 \ .
$$
Then we end up with (\ref{10}).
\end{proof}

Thus, in our approach the function $\phi (\boldsymbol{x})$ plays the role of
producing function. Choosing first $\phi (\boldsymbol{x})$ and then using
its pre-Schwarzian (\ref{6}), we define the operator (\ref{7}) and
immediately get the solution (\ref{9}) following from Theorem 2.1.

Let's give an example of our method. We consider the following 2-dimensional
partial differential equation for $\psi (x,y)$:
\begin{equation}
\psi _{xx}+\psi _{yy}-2\psi _{xy}=0\ .  \label{e1}
\end{equation}%
Now let's choose for the simple producing function
\begin{equation}
\phi (x,y)=\frac{x}{y}  \label{e2}
\end{equation}%
the operator
\begin{equation}
D=\frac{\partial }{\partial x}-\frac{\partial }{\partial y}\ ,  \label{e3}
\end{equation}%
i.e. $a_{1}=1$, $a_{2}=-1$, $x_{1}=x$, and $x_{2}=y$. Then
\begin{equation*}
D\left( \frac{x}{y}\right) =\frac{x+y}{y^{2}}\ \ ;\ \ D^{2}\left( \frac{x}{y}%
\right) =\frac{2(x+y)}{y^{3}}
\end{equation*}%
and
\begin{equation*}
V=\frac{2}{y}\ .
\end{equation*}%
The generalized Weiss operator for (\ref{e3}) is
\begin{eqnarray}
L_{2} &=&\left( \frac{\partial }{\partial x}-\frac{\partial }{\partial y}-%
\frac{1}{y}\right) \left( \frac{\partial }{\partial x}-\frac{\partial }{%
\partial y}+\frac{1}{y}\right) =  \notag \\
&=&\frac{\partial ^{2}}{\partial x^{2}}-2\frac{\partial }{\partial x}\frac{%
\partial }{\partial y}+\frac{\partial ^{2}}{\partial y^{2}}\equiv D^{2}\
.
\label{e4}
\end{eqnarray}
Note it reproduces the differential operator in the left hand side
of (\ref{e1}).

The solution of (\ref{e1}) by Theorem 2.1 is given by
\begin{equation}
\psi (x,y)=\frac{c_{0}+c_{1}\phi }{(D\phi )^{1/2}}=\frac{c_{0}y+c_{1}x}{%
\sqrt{x+y}}\ .  \label{e5}
\end{equation}

\subsection{Linear and non-linear examples}

Whether or not the partial differential equation under consideration is
non-linear, the use of generalized Weiss operators can be seen to be an
effective method in the construction of some class of partial differential
equations. In order to make it clear, we provide the reader with one more
linear example and two non-linear examples of partial differential equations.

In general, the generalized Weiss operator%
\begin{equation*}
L_{2}=\left( D-\frac{V}{2}\right) \left( D+\frac{V}{2}\right)
\end{equation*}%
acting on a continously differentiable function $\psi (x,y)$ yields a
partial differential equation in the form%
\begin{equation}
D^{2}\psi (x,y)+Q(x,y)\psi (x,y)=0,  \label{sll}
\end{equation}%
where%
\begin{equation}
Q=\frac{1}{2}\left( DV-\frac{1}{2}V^{2}\right) =\frac{1}{2}\left[ D\left(
\frac{D^{2}\phi }{D\phi }\right) -\frac{1}{2}\left( \frac{D^{2}\phi }{D\phi }%
\right) ^{2}\right] .  \label{cq}
\end{equation}

In the following, as a linear example that corresponds to (\ref{sll}), we
consider the partial differential equation%
\begin{equation}
\psi _{xx}+x^{4}\psi _{yy}+2x^{2}\psi _{xy}+2x\psi _{y}+\frac{1-2x^{2}}{%
\left( 1+x^{2}\right) ^{2}}\psi =0.  \label{lpde}
\end{equation}%
Choosing for the producing function $\phi =x+y$ the operator%
\begin{equation*}
D=\frac{\partial }{\partial x}+x^{2}\frac{\partial }{\partial y},
\end{equation*}%
we get $D\phi =1+x^{2}$ and $D^{2}\phi =2x$. Then we obtain%
\begin{equation*}
V=\frac{2x}{1+x^{2}}
\end{equation*}%
which yields%
\begin{equation*}
Q\left( x,y\right) =\frac{1-2x^{2}}{\left( 1+x^{2}\right) ^{2}}.
\end{equation*}%
Thus the generalized Weiss operator%
\begin{equation*}
L_{2}=\left( \frac{\partial }{\partial x}+x^{2}\frac{\partial }{\partial y}%
\right) ^{2}+\frac{1-2x^{2}}{\left( 1+x^{2}\right) ^{2}}
\end{equation*}%
reproduces (\ref{lpde}) whose solution according to Theorem 2.1 is given by%
\begin{equation*}
\psi (x,y)=\frac{c_{0}+c_{1}\left( x+y\right) }{\sqrt{1+x^{2}}}.
\end{equation*}

Let us now discuss the following partial differential equation in the form
of (\ref{sll}) which, however, is non-linear:%
\begin{equation}
\psi \psi _{xx}+\psi \psi _{yy}-2\psi \psi _{xy}+\frac{1}{2}\psi _{x}^{2}+%
\frac{1}{2}\psi _{y}^{2}-\psi _{x}\psi _{y}=0.  \label{npde}
\end{equation}%
Next, we consider%
\begin{equation*}
D=-\psi \frac{\partial }{\partial x}+\psi \frac{\partial }{\partial y}
\end{equation*}
which leads to%
\begin{equation*}
D\psi =-\psi \psi _{x}+\psi \psi _{y}
\end{equation*}%
and%
\begin{equation*}
D^{2}\psi =\psi ^{2}\psi _{xx}-2\psi ^{2}\left( \psi _{xy}\right) +\psi
^{2}\psi _{yy}+\psi \psi _{x}^{2}+\psi \psi _{y}^{2}-2\psi \psi _{x}\psi _{y}
\end{equation*}%
for $\psi _{xy}=\psi _{yx}$. Thus (\ref{sll}) can be written as%
\begin{equation}
\psi ^{2}\psi _{xx}-2\psi ^{2}\psi _{xy}+\psi ^{2}\psi _{yy}+\psi \psi
_{x}^{2}+\psi \psi _{y}^{2}-2\psi \psi _{x}\psi _{y}+Q\psi =0.  \label{fde}
\end{equation}%
Here, $Q$ can be specified once we define the producing function $\phi (x,y)$%
. Next, we choose $\phi (x,y)=y-x$ to end up with $D\phi =2\psi $ and $%
D^{2}\phi =-2\psi \psi _{x}+2\psi \psi _{y}$ which together yield $V=\psi
_{y}-\psi _{x}$ for a non-trivial solution (i.e., $\psi \neq 0$). According
to (\ref{cq}), we obtain%
\begin{equation}
Q(x,y)=\frac{1}{2}\psi \psi _{xx}+\frac{1}{2}\psi \psi _{yy}-\psi \psi _{xy}-%
\frac{1}{4}\psi _{y}^{2}+\frac{1}{2}\psi _{x}\psi _{y}-\frac{1}{4}\psi
_{x}^{2}.  \label{qcof}
\end{equation}%
Using (\ref{qcof}) in (\ref{fde}), the non-linear partial differential
equation can be expressed as in (\ref{npde}). Then,%
\begin{equation*}
\psi (x,y)=\frac{\left[ c_{0}+c_{1}\left( y-x\right) \right] ^{2/3}}{2^{1/3}}
\end{equation*}%
can be identified to be a real solution of (\ref{npde}) in accordance with
Theorem 2.1.

Next, we present the second non-linear example for $n=2$ for which the
generalized Weiss operator can be written as%
\begin{equation*}
L_{3}=\prod_{j=0}^{2}\left[ D+\left( j-1\right) V\right] =\left( D-V\right)
D\left( D+V\right)
\end{equation*}%
or%
\begin{equation*}
L_{3}=D^{3}+4QD+2DQ
\end{equation*}%
where $Q$ is given by (\ref{cq}). The operator $L_{3}$ acting on a
continuously differentiable function $\psi (x,y,z)$ leads to a partial
differential equation in the form%
\begin{equation}
D^{3}\psi (x,y,z)+4Q(x,y,z)D\psi (x,y,z)+2\psi (x,y,z)DQ(x,y,z)=0.
\label{ipde}
\end{equation}%
Choosing, for instance,%
\begin{equation}
D=\frac{\partial }{\partial x}+\frac{\partial }{\partial y}+\psi \frac{%
\partial }{\partial z}  \label{op}
\end{equation}%
for a producing function $\phi (x,y,z)=x-y+z$, the partial differential
equation in (\ref{ipde}) becomes non-linear. One can see, for (\ref{op}),
that%
\begin{equation*}
D\psi =\psi _{x}+\psi _{y}+\psi \psi _{z}.
\end{equation*}%
The explicit form of (\ref{ipde}) can be obtained using%
\begin{equation*}
D\phi =\psi ,
\end{equation*}%
\begin{equation*}
D^{2}\phi =D\psi =\psi _{x}+\psi _{y}+\psi \psi _{z},
\end{equation*}%
and therefore%
\begin{equation*}
V=\frac{D^{2}\phi }{D\phi }=\frac{\psi _{x}+\psi _{y}+\psi \psi _{z}}{\psi }
\end{equation*}%
together with (\ref{cq}) in (\ref{ipde}). Following from Theorem 2.1,%
\begin{equation*}
\psi (x,y,z)=\pm \sqrt{c_{0}+c_{1}\left( x-y+z\right) +c_{2}\left(
x-y+z\right) ^{2}}
\end{equation*}%
is given as a solution of (\ref{ipde}).

\section{Discussion}

As we mentioned above, starting from the original work \cite{Weiss86}, the
operator family $L_{n+1}$ is usually used in the analysis of partial
differential equations of solitonic type \cite{Kud93}. Now we can extend the
class of those equations to the case of multi-dimensional spatial variables.
The appropriate choice of the coefficients $a_i$ can fit the necessary form
of operator.

\section{Acknowledgment}

The authors are pleased to express their gratitude for the valuable
suggestions of Dr. Plamen Djakov.

\end{document}